\documentclass[12pt]{article}

\usepackage{enumerate}
\usepackage{amsmath}
\usepackage{amsfonts}
\usepackage{amssymb}
\usepackage{graphicx, color}
\usepackage{hyperref}
\usepackage{epsfig}
\usepackage{epstopdf} 

\usepackage{tikz}
\usetikzlibrary{arrows.meta}
\usetikzlibrary{calc}

\usepackage{tocloft}

\usepackage{float}

\definecolor{darkgreen}{rgb}{0,0.55,0}

\newtheorem{proposition}{Proposition}[section]
\newtheorem{theorem}{Theorem}[section]
\newtheorem{lemma}[theorem]{Lemma}

\newtheorem{remark}[theorem]{Remark}

\newtheorem{definition}{Definition}

\DeclareSymbolFont{AMSb}{U}{msb}{m}{n}
\DeclareMathSymbol{\N}{\mathbin}{AMSb}{"4E}
\DeclareMathSymbol{\Z}{\mathbin}{AMSb}{"5A}
\DeclareMathSymbol{\R}{\mathbin}{AMSb}{"52}
\DeclareMathSymbol{\Q}{\mathbin}{AMSb}{"51}
\DeclareMathSymbol{\I}{\mathbin}{AMSb}{"49}

\numberwithin{equation}{section}

\begin{document}

\title{Symmetry and Rigidity Results for the Mean Field Equation and Hawking Mass on \( \mathbb{S}^2 \)}

\author{{Changfeng Gui\footnote{Department of Mathematics, University of Macau, Taipa, Macao. E-mail:  changfenggui@um.edu.mo.}
\qquad Amir Moradifam\footnote{Department of Mathematics, University of California, Riverside, California, USA. E-mail: moradifam@math.ucr.edu. }}}
\date{\today}

\smallbreak \maketitle

\begin{abstract}

In this paper, we establish symmetry results for solutions of the mean field equation 
\[
\frac{\alpha}{2} \Delta u + e^u - 1 = 0
\]
on \( \mathbb{S}^2 \) for $\frac{1}{3}\leq \alpha \leq 1$, under a geometric condition $(\mathcal{H})$ introduced below. The proofs utilize the Sphere Covering Inequality and incorporate topological arguments on \( \mathbb{S}^2 \). These results are further applied to demonstrate a rigidity property of the Hawking mass for stable constant mean curvature (CMC) spheres, addressing a question posed by Robert Bartnik in 2002. Our results unify and extend previous rigidity results obtained under symmetry or hemispherical balance assumptions, and apply beyond the nearly spherical setting.

\end{abstract}

\vspace{1cm}

\section{Introduction}
Consider the unit sphere $\mathbb{S}^2$. For $u \in H^1(\mathbb{S}^2)$, define the nonlinear functional $J_\alpha(u)$ as follows:

\begin{align}\label{jAlpha}
J_\alpha(u) = \frac{\alpha}{4} \int_{\mathbb{S}^2}|\nabla u|^2 \,d\omega + \int_{\mathbb{S}^2}u \,d\omega - \log \int_{\mathbb{S}^2}e^{u} \,d\omega,
\end{align}
where $d\omega$ is the normalized volume form on $\mathbb{S}^2$ satisfying $\int_{\mathbb{S}^2}d\omega = 1$. The Euler-Lagrange equation of the functional \eqref{jAlpha} on the Sobolev space $H^1(\mathbb{S}^2)$ is given by

\begin{equation*}
\frac{\alpha}{2}\Delta u + \frac{e^u}{\int_{\mathbb{S}^2} e^u \,d\omega} - 1 = 0.
\end{equation*}
Since this equation is invariant under adding a constant, we may normalize the solutions and assume that $\int_{\mathbb{S}^2} e^u \,d\omega = 1$. Consequently, the Euler-Lagrange equation becomes

\begin{align}\label{mainPDE}
\frac{\alpha}{2}\Delta u + e^u - 1 = 0, \quad \text{on} \quad \mathbb{S}^2,
\end{align}
where $\Delta$ is the Laplace-Beltrami operator associated with the metric induced from the Euclidean metric of $\mathbb{R}^3.$

The renowned Moser-Trudinger inequality \cite{Moser-MR0301504} states that $J_\alpha$ is bounded below if and only if $\alpha \geq 1$. Later Onofri \cite{O-MR677001} proved that the optimal lower bound for $J_\alpha$ is zero when $\alpha \geq 1$. Aubin \cite{Aubin2-MR534672} showed that if we consider $J_\alpha$ over the restricted class of functions

\[
\mathcal{M}:=\{u\in H^1(S^2): \int_{S^2}e^u x_i\,d\omega=0, \text{ for } i=1,2,3\},
\]
then $J_\alpha$ is bounded below for $\alpha \geq \frac{1}{2}$, and the minimum is achieved within $\mathcal{M}$. Chang and Yang \cite{CY2-MR925123, CY1-MR908146}, in their study of prescribing Gaussian curvature on $S^2$, proved that the infimum of $J_{\alpha}(u)$ over $u$ in $\mathcal{M}$ is zero for $\alpha$ close to $1$ and conjectured that this result should hold for any $\alpha \geq \frac{1}{2}$.

This conjecture was proven to be true for axially symmetric functions when $\alpha > \frac{16}{25} - \epsilon$ by Feldman, Froese, Ghoussoub, and the first author \cite{FFGG-MR1606461}. Subsequently, it was established for axially symmetric functions in the entire range $\alpha \geq \frac{1}{2}$ by the first author and Wei \cite{GW-MR1760786}, and independently by Lin \cite{Lin1-MR1770683}. However, the case for non-axially symmetric functions remained unresolved.

Ghoussoub and Lin \cite{GL-MR2670931} later proved that Chang and Yang's conjecture holds for $\alpha \geq \frac{2}{3} - \epsilon$ for some $\epsilon > 0$, but the conjecture remained open for $\alpha \geq \frac{1}{2}$. For a detailed history of this problem, refer to Chapter 19 in \cite{GoussoubMoradifam}, a monograph by Ghoussoub and the second author.

In \cite{GM-SCI}, we proved Chang and Yang's conjecture in its full generality by discovering the Sphere Covering Inequality. This inequality states that the total area of two distinct surfaces with Gaussian curvature less than or equal to 1, which are conformal to the same planar domain and have the same conformal factor on the boundary, must be at least $4\pi$, and hence is enough to cover the unit sphere after a rearrangement. In the formulation below, the metric $e^{2v_i}dy$ has Gaussian curvature $K_i=1-e^{-2v_i}f_i$, so the condition $f_i\geq 0$ corresponds precisely to $K_i\leq 1$. 

\begin{theorem}[The Sphere Covering Inequality \cite{GM-SCI}]\label{SCI}
Let $\Omega $ be a simply-connected subset of $\R^2$ and assume $v_i \in C^2(\overline{\Omega})$, $i=1,2$ satisfy
\begin{align}\label{mainpde}
\Delta v_i +e^{2v_i}=f_{i}(y),
\end{align}
where $f_2 \geq f_1 \ge 0$ in $\Omega$.  If $v_2 \ge v_1, v_2 \not \equiv v_1$ in $\omega$ and $v_2=v_1$ on $\partial \omega$ for some piecewise
Lipschitz subdomain $\omega \subset \Omega$, then
\begin{align}
\int_{\omega} (e^{2v_1}+e^{2v_2})dy \geq 4\pi. 
\end{align}
Moreover, the equality only holds when $f_2 \equiv f_1 \equiv 0$ in $\omega$, and $(\omega, e^{2v_i}dy ) $, $i=1,2$ are isometric to two complementary spherical caps on the standard unit sphere. 
\end{theorem}

The dual of the Sphere Covering Inequality was also discovered by the authors and Hang in \cite{GHM-Dual}. The Sphere Covering Inequality was the missing piece in many uniqueness and symmetry problems related to mean field equations, and it has been applied to solve various open questions in related problems \cite{MR3985126, GHMX-Hawking, MR3804203, GM-Torus, MR3750235, Tian}.

In 2002, at the International Congress of Mathematicians (ICM) held in Beijing, Robert Bartnik introduced the rigidity problem for Hawking mass; see page 235 in \cite{MR1957036} and Theorem \ref{SunLemma} below. For background on the concept of Hawking mass in general relativity, refer to \cite{MR3539922, MR3960907, MR2167257, MR2320089}. Sun \cite{MR3733981} established a connection between rigidity of Hawking mass and uniqueness of solutions of the mean field equation \eqref{mainPDE} for $\alpha=\frac{1}{3}$. More precisely, after normalizing $|\Sigma|=4\pi$, the equality case in the stability argument yields a conformal map from $\Sigma$ to the standard sphere and the induced metric can be written as
\[
g_\Sigma=e^u g_0,
\]
where $g_0$ is the standard round metric on $\mathbb{S}^2$. The conformal factor $u$ satisfies
\[
\Delta_{g_0}u=6-6e^u,
\]
together with the volume-preserving conditions
\[
\int_{\mathbb{S}^2}x_i e^u\,d\omega=0,\qquad i=1,2,3.
\]
Thus the equation above is exactly \eqref{mainPDE} with $\alpha=\frac{1}{3}$. We emphasize that this equation does not arise by assuming that $\Sigma$ has constant Gaussian curvature; rather, it follows from the equality case of the stability and spectral estimates in \cite{MR3733981}. Indeed, under the conformal change one has $K_\Sigma=e^{-u}\left(1-\frac12\Delta_{g_0}u\right)$. Sun's reduction can be summarized in the following theorem.

\begin{theorem}[\cite{MR3733981}]\label{SunLemma}
Let $(M,g)$ be a complete Riemannian three-manifold with scalar curvature $R(g) \geq 0$ and let $\Omega \subset M$ be a domain with boundary $\Sigma = \partial\Omega$. We further assume that $\Sigma$ is a stable CMC sphere with $m_H(\Sigma) = 0$, where $m_H(\Sigma)$ denotes the Hawking mass of $\Sigma$. Then $\Omega$ is isometric to a Euclidean ball in $\mathbb{R}^3$ whenever the mean field equation
\begin{align}\label{mainPDE-1/3}
\frac{1}{6}\Delta u + e^u-1=0 \quad \text{on } \ \ \mathbb{S}^2
\end{align}
has only the zero solution.
\end{theorem}

Using the above lemma, Sun \cite{MR3733981} showed that the only sufficiently small solution for equation \eqref{mainPDE-1/3} is zero, and hence established the rigidity of the Hawking mass for surfaces that are nearly spherical.

In \cite{Tian}, the authors proved rigidity of Hawking mass under the even symmetry assumption, i.e. there exists an isometry $\rho \colon \Sigma \to \Sigma$ satisfying: $\rho^2 = \text{id}$ and $\rho(x) \neq x$ for $x \in \Sigma$, or equivalently under the assumption that $u(x)=u(-x)$ for all $x\in \mathbb{S}^2$, where $u$ is a solution of \eqref{mainPDE-1/3}. Indeed, they proved the following theorem. 

\begin{theorem}[\cite{Tian}]\label{Theorem1}
 Let \( \frac{1}{3} \leq \alpha < 1 \) and assume \( u \) is a solution of \eqref{mainPDE}, with \( u(x) = u(-x) \). Then \( u \) must be axially symmetric. Moreover, if \( \alpha = \frac{1}{3} \), then \( u = 0 \).
\end{theorem}

In \cite{GHMX-Hawking}, the authors proved the rigidity of Hawking mass under a weaker assumption. Define 
\[
H_y = \{ x \in \mathbb{S}^2 : x \cdot y \geq 0 \}.
\]

\begin{theorem}[\cite{GHMX-Hawking}]\label{Theorem2}
Let $\frac{1}{3}\leq \alpha \leq 1$, and \( u \) be a solution to \eqref{mainPDE}. If
\begin{eqnarray}\label{IntegralSymmetry}
\int_{H_y}e^{u}dS=\int_{H_{-y}}e^{u}dS \ \ \text{for all} \ \ y\in \mathbb{S}^2,
\end{eqnarray}
then \( u \) is axially symmetric with respect to some point. Moreover, if \( \alpha = \frac{1}{3} \), then \( u \equiv 0 \). 
\end{theorem}

In this paper, we prove rigidity of Hawking mass under the following weaker assumption. Geometrically, condition $(\mathcal{H})$ says that on every great circle $C_y$, the directional derivative of $u$ in the normal direction $y$ vanishes at least at one point:

\begin{center}
\textbf{($\mathcal{H}$)} \hspace{.5cm} For every \( y \in \mathbb{S}^2 \), there exists \( x \in C_y \) such that \( \nabla u(x) \cdot y = 0 \).
\end{center}
where
\[
C_y=\partial H_y=\{x\in\mathbb{S}^2:x\cdot y=0\}.
\]
If $R_yx=x-2(x\cdot y)y$ denotes reflection with respect to the plane orthogonal to $y$, we write $\tilde u(x)=u(R_yx)$. \\ \\

The main result of this paper is the following theorem. 

\begin{theorem} \label{mainTheorem} Let $\frac{1}{3}\leq \alpha \leq 1 $, and $u$ be a solution of \eqref{mainPDE} satisfying $(\mathcal{H})$. Then $u$ is axially symmetric. Moreover, if \( \alpha = \frac{1}{3} \), then \( u \equiv 0 \). 
\end{theorem}

\begin{remark}
Note that Theorem \ref{mainTheorem} includes Theorems \ref{Theorem1} and \ref{Theorem2} as special cases. Indeed, if \( u(x) = u(-x) \) for all \( x \in \mathbb{S}^2 \), then the assumption \eqref{IntegralSymmetry} follows immediately. Moreover, if \eqref{IntegralSymmetry} holds for some \( y \), then:
\begin{align}
0 &= \int_{H_y} e^{u} \, dS - \int_{H_{-y}} e^{u} \, dS \\
  &= \int_{H_y}\big(e^u-e^{\tilde u}\big)\,dS \\
  &= -\frac{\alpha}{2}\int_{H_y}\big(\Delta u-\Delta\tilde u\big)\,dS \\
  &= \frac{\alpha}{2}\int_{C_y}\nabla(u-\tilde u)\cdot y\,ds \\
  &= \alpha\int_{C_y}\nabla u\cdot y\,ds,
\end{align}
where in the fourth equality we used that the outward unit conormal to $H_y$ along $C_y$ is $-y$.
Therefore, \( \nabla u(x) \cdot y \) changes sign on \( C_y \), and hence it must have at least two zeros. Thus the hypothesis \((\mathcal{H})\) is satisfied.

\end{remark}

The proof of Theorem \ref{mainTheorem} relies on the Sphere Covering Inequality and the Hairy Ball Theorem applied to a continuous tangent vector field on $\mathbb{S}^2$ that we will construct in the next section.

\section{Symmetry about a plane}

In this section, we prove that if \( u \) satisfies the hypothesis \( (\mathcal{H}) \), then it must be symmetric about the plane passing through the origin with normal vector \( x \in \mathbb{S}^2 \). We will demonstrate that if \( u \) is not symmetric about any plane passing through the origin, then there exists a continuous non-vanishing tangent vector field \( \mathbb{F}(x) \) on \( \mathbb{S}^2 \), which would violate the Hairy Ball Theorem.

\begin{definition}\label{C(x)}
 For any \( x \in \mathbb{S}^2 \), we define \( C_x \) to be the great circle obtained from the intersection of the plane passing through the origin with the normal vector \( x \), denoted by \( P_x \), and \( \mathbb{S}^2 \). The curve \( C_x \) is said to have a positive orientation if it rotates counter-clockwise when viewed from point \( x \) while traveling along \( C_x \). The plane \( P_x \) divides \( \mathbb{S}^2 \) into two hemispheres, and by \( \mathbb{S}^2_{+}(x) \) we denote the hemisphere containing \( x \). Similarly, \( \mathbb{S}^2_{-}(x) \) denotes the hemisphere that does not contain \( x \).

\end{definition}

\begin{lemma}\label{threePointLemma}
Let $\frac{1}{3}\leq \alpha \leq 1 $. If there are three distinct points $p_i\in C_{x_0}$ for some $x_0\in \mathbb{S}^2$ such that 
\begin{align}\label{threePoint}
\nabla u (p_i) \cdot x_0=0, \ \ i=1,2,3 
\end{align}
then $u$ is evenly symmetric about $P_{x_0}$, i.e. $u(x)=u(\tilde{x})$, where $\tilde{x}$ is the reflection of $x$ with respect to the plane $P_{x_0}$.
\end{lemma}
{\bf Proof.} Suppose there are three distinct points $p_i\in \mathbb{S}^2$, $i=1,2,3$, such that \eqref{threePoint} holds. Without loss of generality we may assume that $p_1, p_2, p_3$ all lie on the $x_1x_3$-plane and $p_i\neq N:=(0,0,1)$ for $i=1,2,3$. Now let $\Pi$ be the stereographic projection $S^2\rightarrow \R^2$ with respect to the north pole $N=(0,0, 1)$: 

\[\Pi:= \left( \frac{x_1}{1-x_3}, \frac{x_2}{1-x_3}\right).\]
Let $u$ be a solution of \eqref{mainPDE} and 
\[\bar{u}(y):=u(\Pi^{-1}(y)) \ \ \hbox{for}\ \ y=(y_1, y_2)\in \R^2.\]
Then $\bar{u}$ satisfies 
\begin{align}
\Delta \bar{u}+\frac{8}{\alpha(1+|y|^2)^2}(e^{\bar{u}}-1)=0 \ \ \hbox{in} \ \ \R^2. 
\end{align}
Now set
\begin{align}
v:=\bar{u}-\frac{2}{\alpha} \ln (1+|y|^2)+\ln (\frac{8}{\alpha}),
\end{align}
then $v$ satisfies 

\begin{equation*}
\Delta v+(1+|y|^2)^{2(\frac{1}{\alpha}-1)}e^{v}=0 \ \ \hbox{in} \ \ \R^2,
\end{equation*}
and 
\begin{equation*}
\int_{\mathbb{R}^2}(1+|y|^2)^{2(\frac{1}{\alpha}-1)}e^{v} dy=\frac{8\pi}{\alpha}. 
\end{equation*}
Let $w=\ln ((1+|y|^2)^{2(\frac{1}{\alpha}-1)}e^{v})$. Then $w$ satisfies 
\begin{align}\label{w}
\Delta w+e^w=\frac{ 8\left( \frac{1}{\alpha} - 1 \right)}{ \left(1+|y|^2\right)^2}\geq 0
\end{align}
with 
\begin{align}\label{Total}
\int_{\mathbb{R}^2}e^wdy=\frac{8\pi}{\alpha}.
\end{align}
The three distinct points $p_1,p_2,p_3$ are mapped to three points $P_1, P_2, P_3 \in \mathbb{R}^2$, lying on $y_1$-axis, with 
\begin{align}
\nabla w (P_i)\cdot (0,1)=0. 
\end{align}
Now let $\varphi(y)=w(y)-w(\tilde{y})$, where $\tilde{y}=(y_1,-y_2)$. Since $\varphi(y)$ satisfies the linear elliptic equation 
\begin{align}\label{linearElliptic}
\Delta \varphi+\left(\frac{e^{w(y)}-e^{w(\tilde{y})}}{w(y)-w(\tilde{y})}\right)\varphi=0
\end{align}
we have $\varphi(P_i)=0$ and $\nabla\varphi(P_i)=0$. If the $y_1$-axis were the only nodal arc near $P_i$, the Hopf Lemma applied to either adjacent nodal domain would imply $\partial_{y_2}\varphi(P_i)\neq0$, a contradiction. Hence another nodal arc enters the upper half-plane from $P_i$, and by reflection there is a corresponding nodal arc in the lower half-plane. Equivalently, by the local structure theorem for nodal sets, near each $P_i$ we have $\varphi(y)=Q_i(y-P_i)+$ higher-order terms, where $Q_i$ is a nonzero harmonic homogeneous polynomial of degree $m_i\geq2$. Thus the nodal set divides a neighborhood of each $P_i$ into at least four regions. Therefore, the nodal line of $\varphi$ must divide $\mathbb{R}^2_{+}$ into at least three regions $\Omega_1,\Omega_2, \Omega_3$. Now notice that on each $\Omega_i$ the equation \eqref{w} has two solutions $w_1(y)=w(y)$ and $w_2(y)=w(\tilde{y})$ with $w_1=w_2$ on $\partial \Omega_i$, $i=1,2,3$. Hence it follows from the Sphere Covering Inequality that 
\[
\frac{8\pi}{\alpha}=\int_{\mathbb{R}^2}e^{w}dy
\geq \sum_{i=1}^{3}\int_{\Omega_i}\big(e^{w(y)}+e^{w(\tilde{y})}\big)dy
\geq 24\pi.
\]
This is impossible when $\alpha>\frac13$. If $\alpha=\frac13$, equality would have to hold in each application of the Sphere Covering Inequality, which is impossible because the right-hand side of \eqref{w} is then strictly positive. Thus we again obtain a contradiction. Therefore, $u$ must be symmetric about $P_{x_0}$. \hfill $\Box$ 
\vspace{.4 cm}

\subsection{A non-vanishing tangent vector field on $\mathbb{S}^2$}

In this subsection, assuming the hypothesis \( (\mathcal{H}) \) and the condition that a solution \( u \) of Equation \eqref{mainPDE} is not symmetric about any plane passing through the origin, we construct a tangent vector field on \( \mathbb{S}^2 \). First note that, by Lemma \ref{threePointLemma}, for every $x\in \mathbb{S}^2$ there are at most two points $p_1,p_2$ on $C_x$ such that 
\[\nabla u (p_i)\cdot x=0, \ \ i=1,2.\]
Hence 
\begin{align}\label{DisjointSets}
\mathbb{S}^2=\mathcal{S}_2\cup \mathcal{S}_1,
\end{align}
where $\mathcal{S}_i$ is the set of all points $x \in \mathbb{S}^2$ for which there are exactly $i$ points $p\in C_x$ where the gradient of $u$ at $p$ is orthogonal to $x$, i.e. $\nabla u(p)\cdot x=0.$ Define the tangent vector field $\mathbb{F}$ on $\mathbb{S}^2$ as follows. \\ \\
\textit{Case I}. Let $x\in \mathcal{S}_2$. Then there exist exactly two distinct points $p_1, p_2 \in C_x$ such that $\nabla u(p_i)\cdot x=0 $, $i=1,2$. Without loss of generality we may assume that $p_1$ is the point where $\nabla u(p)\cdot x$ changes sign from $+$ to $-$ while $p$ moves in counterclockwise direction, and $p_2$ is the point where the sign changes from $-$ to $+$. Now define $\mathbb{F}$ to be a unit tangent vector on $\mathbb{S}^2$ at $x$ pointing in the direction of $\frac{p_2-p_1}{|p_2-p_1|}$. \\ \\ 
\textit{Case II. } Let $x\in \mathcal{S}_1$. Then there exists only one point $p \in C_x$ such that $\nabla u(p)\cdot x=0 $. If \(\nabla u \cdot x \geq 0\) on \(C_x\), then let \(\mathbb{F}\) be the unit tangent vector to \(C_x\) at point \(p\), pointing in the counterclockwise direction. Similarly, if \(\nabla u \cdot x \leq 0\) on \(C_x\), then let \(\mathbb{F}\) be the unit tangent vector to \(C_x\) at point \(p\), pointing in the clockwise direction. \\ \\

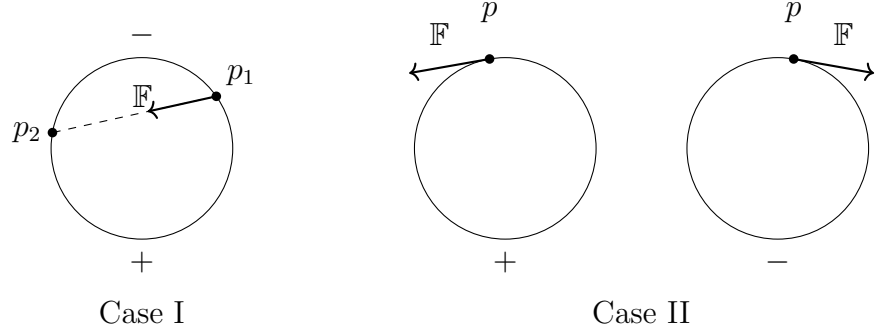
\begin{figure}[H]
\centering
\begin{tikzpicture}[scale=1.2]

\def\r{1}

\coordinate (C1) at (0,0);
\draw (C1) circle (\r);

\coordinate (pone) at ({\r*cos(35)},{\r*sin(35)});
\coordinate (ptwo) at ({\r*cos(170)},{\r*sin(170)});

\fill (pone) circle (1.5pt);
\node[above right] at (pone) {$p_1$};

\fill (ptwo) circle (1.5pt);
\node[left] at (ptwo) {$p_2$};

\draw[dashed] (pone) -- (ptwo);

\draw[->, thick] (pone) -- ($(pone)!0.42!(ptwo)$);
\node at ($(pone)!0.45!(ptwo)+(0,0.18)$) {$\mathbb{F}$};

\node at (0,1.25) {$-$};   
\node at (0,-1.25) {$+$};  

\node at (0,-1.8) {Case I};

\coordinate (C2) at (4,0);
\draw (C2) circle (\r);

\coordinate (pleft) at ($(C2)+({\r*cos(100)},{\r*sin(100)})$);
\fill (pleft) circle (1.5pt);
\node[above] at ($(pleft)+(0,0.28)$) {$p$};

\draw[->, thick] (pleft) -- ++({-0.9*sin(100)},{0.9*cos(100)});
\node at ($(pleft)+(-0.55,0.28)$) {$\mathbb{F}$};

\node at (4,-1.25) {$+$};

\coordinate (C3) at (7,0);
\draw (C3) circle (\r);

\coordinate (pright) at ($(C3)+({\r*cos(80)},{\r*sin(80)})$);
\fill (pright) circle (1.5pt);
\node[above] at ($(pright)+(0,0.28)$) {$p$};

\draw[->, thick] (pright) -- ++({0.9*sin(80)},{-0.9*cos(80)});
\node at ($(pright)+(0.55,0.28)$) {$\mathbb{F}$};

\node at (7,-1.25) {$-$};

\node at (5.5,-1.8) {Case II};

\end{tikzpicture}

\caption{Definition of the continuous non-vanishing vector field $\mathbb{F}$}
\end{figure}

In the next subsection, we prove that the vector field \( \mathbb{F} \) defined above on \( \mathbb{S}^2 \) is indeed continuous.

\subsection{Continuity of $\mathbb{F}$}
Let $u$ be a solution of \eqref{mainPDE}. For $x_0\in \mathbb{S}^2$ define 
\begin{align}\label{Phi}
\Phi_{x_0}(x):=u(x)-u(\hat{x}), \ \ x\in \mathbb{S}^2,
\end{align}
where $\hat{x}$ is the reflection of $x$ with respect to the plane $P_{x_0}$. Also define the nodal set of $\Phi_{x_0}(x)$ as follows
\begin{align}\label{nodalSet}
\mathcal{N}_{x_0}:=\{x\in \mathbb{S}^2: \Phi_{x_0}(x)=0\}.
\end{align}
The following proposition will play a crucial role in the analysis of the nodal set $\mathcal{N}_{x_0}$. 

\begin{proposition}\label{OnlyOneOmega}
Suppose $\frac{1}{3}\leq \alpha \leq 1$, and let $u$ be a solution of \eqref{mainPDE} that is not symmetric about any plane passing through the origin. 
Then $\mathcal{N}_{x_0}$ divides $\mathbb{S}^2$ into two pairs of open sets $\Omega_i, \tilde{\Omega}_{i}$, $i=1,2$, such that 
\begin{align}\label{4regions}
\mathbb{S}^2=\mathcal{CL}( \Omega_1\cup \tilde{\Omega}_{1}\cup \Omega_2 \cup \tilde{\Omega}_{2} )
\end{align}
where $\tilde{\Omega}_{i}$ is the reflection of $\Omega_{i}$ with respect to the plane $P_{x_0}$, and $\mathcal{CL}$ is the closure set operator. Moreover, $\nabla u(p)\cdot x_0=0$ for $p\in C_{x_0}$ if and only if 
\begin{align}\label{normalDerivativeCondition}
p\in C_{x_0}\cap \partial \Omega_1 \cap \partial \Omega_2. 
\end{align}
\end{proposition}
{\bf Proof.} First choose $x_0$ so that a critical point $p$ of $u$ lies on $C_{x_0}$. Since $\nabla u(p)\cdot x_0=0$ and $C_{x_0}\subset \mathcal{N}_{x_0}$, it follows from an argument similar to the one used in the proof of Lemma \ref{threePointLemma} that the nodal line $\mathcal{N}_0$ of $\Phi_{x_0}$ divides a neighborhood of $p$ into at least four regions, and hence it divides $\mathbb{S}^2$ into exactly four regions satisfying \eqref{4regions}. Note that if $\mathcal{N}_0$ divides $\mathbb{S}^2$ into more than four regions, then it will divide $\mathbb{S}^2$ into at least three pairs of subsets of $\mathbb{S}^2$, $(\Omega_i, \tilde{\Omega}_i)$ ($i=1,2,3$), and hence with an argument similar to that in the proof of Lemma \ref{threePointLemma}, the Sphere Covering Inequality implies that $u$ is symmetric about $P_{x_0}$. Thus \eqref{4regions} holds if $C_{x_0}$ contains a critical point of $u$. 

Let $\mathcal{A}$ be the set of points in $\mathbb{S}^2$ for which \eqref{4regions} holds. By the above argument we know that $\mathcal{A}\neq \emptyset$. We next show that $\mathcal{A}$ is open. Fix $x_0\in\mathcal{A}$ and choose one point in each of the two nodal domains contained in $\mathbb{S}^2_+(x_0)$. The values of $\Phi_{x_0}$ at these two points have opposite signs. After identifying nearby hemispheres by a small rotation, the corresponding values of $\Phi_x$ retain these signs for $x$ sufficiently close to $x_0$. Hence $\Phi_x$ has at least two nodal domains in $\mathbb{S}^2_+(x)$ and, by reflection, at least four nodal domains on $\mathbb{S}^2$. If there were more than four, the Sphere Covering Inequality would imply symmetry about $P_x$, contrary to our assumption. Thus \eqref{4regions} persists for all $x$ sufficiently close to $x_0$, and $\mathcal{A}$ is open.

Next we show that $\mathcal{A}$ is also closed. Let $x_n\in\mathcal{A}$ converge to $x_0\in\mathbb{S}^2$. Without loss of generality, after rotating coordinates, we may assume that $C_{x_0}$ lies in the plane $x_1=0$. With an argument similar to that in the proof of Lemma \ref{threePointLemma} and using stereographic projection, we obtain a function $w:\mathbb{R}^2\rightarrow\mathbb{R}$ satisfying \eqref{w}. If $\Omega_i$ is one of the two nodal domains in a hemisphere and $\omega_i$ is its stereographic image, then $w_1(y)=w(y)$ and $w_2(y)=w(\tilde y)$ both satisfy \eqref{w} on $\omega_i$ and agree on $\partial\omega_i$. Hence the Sphere Covering Inequality gives
\[
\int_{\omega_i}\big(e^{w(y)}+e^{w(\tilde y)}\big)\,dy\geq 8\pi.
\]
Since
\[
e^w\,dy=\frac{2}{\alpha}e^u\,dS,
\]
and $u$ is bounded on $\mathbb{S}^2$, it follows that the spherical area of each $\Omega_i$ is bounded from below by a positive constant independent of $x_n$; namely,
\begin{align}\label{AreaLowerBound}
\mu(\Omega_i)\geq c,\qquad i=1,2,
\end{align}
for some $c>0$. Thus none of the four nodal regions can collapse in the limit $x_n\to x_0$. Indeed, $\Phi_{x_n}\to\Phi_{x_0}$ uniformly, and if one of the two signs disappeared in the limiting hemisphere, the corresponding nodal domains of uniformly positive area would accumulate on a subset of positive area on which $\Phi_{x_0}=0$. By the local structure theorem (or unique continuation) for the linear elliptic equation satisfied by $\Phi_{x_0}$, this would force $\Phi_{x_0}\equiv0$, contrary to the assumption that $u$ is not symmetric about $P_{x_0}$. Hence the limiting nodal set has at least two nodal domains in each hemisphere, and therefore at least four on $\mathbb{S}^2$. If it had more than four, the Sphere Covering Inequality would again imply symmetry about $P_{x_0}$. Consequently \eqref{4regions} also holds at $x_0$, and $\mathcal{A}$ is closed. Hence $\mathcal{A}\neq \emptyset$ is both open and closed, and therefore
\[\mathcal{A}=\mathbb{S}^2.\]
Thus \eqref{4regions} holds for all $x_0 \in \mathbb{S}^2.$ To prove \eqref{normalDerivativeCondition}, let $\varphi=u(x)-u(\tilde{x})$, where $\tilde{x}$ is the reflection of $x$ with respect to $P_{x_0}$. Then $\varphi$ satisfies 
\[\frac{\alpha}{2}\Delta \varphi(x)+\frac{e^{u(x)}-e^{u(\tilde{x})}}{u(x)-u(\tilde{x})}\varphi (x)=0 \ \ x\in \mathbb{S}^2_{+}(x_0),\]
and $\varphi=0$ on $C_{x_0}$. It is now easy to see that \eqref{normalDerivativeCondition} follows from the Hopf lemma. \hfill $\Box$

\vspace {.4cm}

\vspace {.4cm}

We shall also need the following result. 
\begin{lemma}\label{BoundaryofSs}
Suppose $\frac{1}{3}\leq \alpha \leq 1$, and let $u$ be a solution of \eqref{mainPDE} satisfying $(\mathcal{H})$ that is not symmetric about any plane passing through the origin. Also let $\mathcal{S}_1, \mathcal{S}_2$ be defined as in \eqref{DisjointSets}. Then $\mathcal{S}_2$ is an open subset of $\mathbb{S}^2$ and $\mathcal{S}_1$ is closed. In particular, 
\[ \partial \mathcal{S}_2\subset \mathcal{S}_1.\]
\end{lemma}
{\bf Proof.} Let $x_0\in\mathcal{S}_2$, and let $p_1,p_2\in C_{x_0}$ be the two zeros of $\nabla u\cdot x_0$. By Proposition \ref{OnlyOneOmega} and the Hopf lemma, $\nabla u\cdot x_0$ has opposite signs on the two arcs of $C_{x_0}\setminus\{p_1,p_2\}$. Choose one point on each arc where these signs are strict. After identifying nearby great circles by a small rotation, continuity of $\nabla u$ implies that the corresponding two values of $\nabla u\cdot x$ retain opposite signs for $x$ sufficiently close to $x_0$. The intermediate value theorem then gives at least two zeros on $C_x$. By Lemma \ref{threePointLemma}, there can be at most two. Hence $x\in\mathcal{S}_2$ for all $x$ sufficiently close to $x_0$, and $\mathcal{S}_2$ is open. In view of \eqref{DisjointSets}, $\mathcal{S}_1=\mathbb{S}^2\setminus\mathcal{S}_2$, so $\mathcal{S}_1$ is closed. \hfill $\Box$

In this section, we prove that the vector field \( \mathbb{F} \), constructed in Section 2.1, is indeed continuous.

\begin{proposition}\label{continuityProp}
Suppose $\frac{1}{3} \leq \alpha \leq 1$. Let $u$ be a solution of \eqref{mainPDE} satisfying $(\mathcal{H})$ that is not symmetric about any plane passing through the origin, and $\mathbb{F}$ be the vector field defined in Section 2. The following statements hold. 

\begin{enumerate}[(i)]
 \item If $x_n \in \mathcal{S}_i$ converges to some $x_0 \in \mathcal{S}_i$ ($i=1,2$), then $\lim_{n\rightarrow \infty} \mathbb{F}(x_n)=\mathbb{F}(x_0)$. 
 \item If $x_n \in \mathcal{S}_2$ converges to some $x_0 \in \mathcal{S}_1$, then $\lim_{n\rightarrow \infty} \mathbb{F}(x_n)=\mathbb{F}(x_0)$. 
\end{enumerate}
Moreover, the vector field $\mathbb{F}$ is continuous on $\mathbb{S}^2$. 
\end{proposition}
{\bf Proof.} $(i)(a)$ Suppose $x_n \in \mathcal{S}_2$ converges to some $x_0 \in \mathcal{S}_2$. Then for each $n \in \mathbb{N}$, there exist exactly two points $p_1^n$ and $p_2^n$ on $C_{x_n}$ such that $\nabla u(p_i^n)\cdot x_n=0$, and $p_1$ is the point where $\frac{\partial u}{\partial x_n} (p)$ changes sign from positive to negative as $p$ moves counterclockwise, and $p_2$ is the point where the sign changes from negative to positive with
\[\mathbb{F}(x_n)=\frac{p_2^n-p_1^n}{|p_2^n-p_1^n|}.\]
Since $x_0 \in \mathcal{S}_2$, there are also exactly two points $p_1$ and $p_2$ on $C_{x_0}$ such that $\nabla u(p_i)\cdot x_0=0$, and
\[\mathbb{F}(x_0)=\frac{p_2-p_1}{|p_2-p_1|}.\]
By compactness, every subsequence of $\{p_i^n\}$ has a further convergent subsequence, and continuity of $\nabla u$ shows that every such limit is a zero of $\nabla u\cdot x_0$ on $C_{x_0}$. The sign-change labeling of $p_1^n,p_2^n$ prevents the two limits from being interchanged. Hence $p_i^n\to p_i$, $i=1,2$, and therefore
\[\lim_{n\rightarrow \infty}\mathbb{F}(x_n)=\mathbb{F}(x_0).\]

\vspace{.3cm}
\noindent 
$(i)(b)$ Suppose \(x_n \in \mathcal{S}_1\) converges to some \(x_0 \in \mathcal{S}_1\). Then, for each \(n \in \mathbb{N}\), there exists exactly one point \(p^n\) on \(C_{x_n}\) such that \(\nabla u(p^n) \cdot x_n = 0\), and \(\mathbb{F}(x_n)\) is the unit tangent vector \(V_n\) to \(C_{x_n}\) if \(\nabla u \cdot x_n \geq 0\), and \(\mathbb{F}(x_n) = -V_n\) otherwise. Similarly, since \(x_0 \in \mathcal{S}_1\), there is exactly one point \(p\) on \(C_{x_0}\) such that \(\nabla u(p) \cdot x_0 = 0\), and \(\mathbb{F}(x_0)\) is the unit tangent vector \(V_0\) to \(C_{x_0}\) at \(p\) if \(\nabla u \cdot x_0 \geq 0 \), and \(\mathbb{F}(x_0) = -V_0\) otherwise. By compactness and continuity, every convergent subsequence of \(\{p^n\}\) has a limit at which \(\nabla u\cdot x_0=0\). Since \(x_0\in\mathcal{S}_1\), this limit can only be \(p\). Thus \(p^n\to p\), and therefore \(\lim_{n \to \infty} \mathbb{F}(x_n) = \mathbb{F}(x_0).\) 
\vspace{.3cm}

$(ii)$ Suppose \(x_n \in \mathcal{S}_2\) converges to some \(x_0 \in \mathcal{S}_1\). Similar to the case $(i)(a)$, for each $n \in \mathbb{N}$, there exist exactly two points $p_1^n$ and $p_2^n$ on $C_{x_n}$ such that $\nabla u(p_i^n)\cdot x_n=0$, and $p_1^n$ is the point where $\nabla u(p)\cdot x$ changes sign from positive to negative while $p$ moves in counterclockwise direction, and $p_2^n$ is the point where the sign changes from negative to positive with
 \[\mathbb{F}(x_n)=\frac{p_2^n-p_1^n}{|p_2^n-p_1^n|}.\] 
Since \(x_0 \in \mathcal{S}_1\), there is exactly one point \(p\) on \(C_{x_0}\) such that \(\nabla u(p) \cdot x_0 = 0\), and \(\mathbb{F}(x_0)\) is the unit tangent vector \(V_0\) to \(C_{x_0}\) at $p$ if \(\nabla u \cdot x_0 \geq 0 \), and \(\mathbb{F}(x_0) = -V_0\), otherwise. Both the sequences \(\{p_1^n\}\) and \(\{p_2^n\}\) converge to \(p\), as otherwise there would exist a second point \(p'\) on $C_{x_0}$ where \(\nabla u(p') \cdot x_0 = 0\), contradicting the assumption that \(x_0 \in \mathcal{S}_1\). Parameterize $C_{x_n}$ near $p$ by the positively oriented arc-length parameter. The labeling means that the sign of $\nabla u\cdot x_n$ on the oriented arc from $p_1^n$ to $p_2^n$ is opposite to its sign on the complementary arc. Since the limiting function $\nabla u\cdot x_0$ has no zero away from $p$, the first of these arcs must shrink to $p$ (after interchanging the two arcs if necessary according to the sign of the limiting function). Thus $p_1^n$ and $p_2^n$ approach $p$ from opposite sides, and the normalized chord $(p_2^n-p_1^n)/|p_2^n-p_1^n|$ converges to the unit tangent vector to $C_{x_0}$ at $p$, with its sign determined by the labeling of $p_1^n$ and $p_2^n$. Therefore,
\[
\lim_{n \rightarrow \infty} \mathbb{F}(x_n)=
\begin{cases}
V_0 & \text{if} \ \ \nabla u \cdot x_0 \geq 0 \ \ \text{on} \ \ C_{x_0},\\
-V_0 & \text{if} \ \ \nabla u \cdot x_0 \leq 0 \ \ \text{on} \ \ C_{x_0},
\end{cases}
\]
and hence \(\lim_{n \to \infty} \mathbb{F}(x_n) = \mathbb{F}(x_0).\)

In view of Lemma \ref{BoundaryofSs}, it follows from statements $(i)$ and $(ii)$ that the vector field $\mathbb{F}$ is continuous on $\mathbb{S}^2$. \hfill $\Box$
\vspace{.5cm}

\begin{proposition}\label{FirstPlaneProp}
Suppose $\frac{1}{3} \leq \alpha \leq 1$. Let $u$ be a solution of \eqref{mainPDE} satisfying $(\mathcal{H})$. Then $u$ is symmetric about a plane passing through the origin. 
\end{proposition}
{\bf Proof.} Suppose \( u \) is not symmetric about any plane. Then, by Proposition \ref{continuityProp}, the non-vanishing tangent vector field \( \mathbb{F} \) defined in Subsection 2.1 is continuous on \( \mathbb{S}^2 \), which contradicts the Hairy Ball Theorem. Therefore, \( u \) must be symmetric about a plane passing through the origin. \hfill $\Box$

\section{Symmetry about a second plane}
In this section we prove that, for \( \frac{1}{3} \leq \alpha \leq 1 \), if a solution of the equation \eqref{mainPDE} is symmetric with respect to a plane passing through the origin, then it must be symmetric with respect to two orthogonal planes.

\begin{lemma} \label{secondPlaneLemma2} Let \( \frac{1}{3} \leq \alpha \leq 1 \), and let \( u \) be the solution of equation \eqref{mainPDE}. If \( u \) has two antipodal critical points $\pm x_c$, then it is symmetric with respect to two orthogonal planes passing through the origin and $\pm x_c$. 
\end{lemma}
{\bf Proof.} We may assume that \((0,0,\pm1)\) are the critical points. Using the stereographic projection and an argument similar to that in the proof of Lemma \ref{threePointLemma}, we obtain a function \( w: \mathbb{R}^2 \rightarrow \mathbb{R} \) satisfying the equation \eqref{w}. Furthermore, the origin \( (0,0) \) is a critical point of \( w \). Near the origin, \( w(y_1,y_2) \) can be expressed as
\[
w(y_1,y_2) = \alpha_{0,0} + \alpha_{2,0}y_1^2 + \alpha_{0,2}y_2^2 + \text{higher-order terms}
\]
in a suitable orthogonal basis of $\mathbb{R}^2$. 

Now let $\tilde{y}$ be the reflection of $y$ with respect to $y_1$-axis. Then $\varphi(y)=w(y)-w(\tilde{y})$ solves the linear elliptic equation \eqref{linearElliptic}, and 
\[\varphi(y)=P_{m}+\text{higher-order terms}, \ \ \text{for some} \ \ m\geq 3,\]
where $P_m$ is a harmonic homogeneous polynomial. Hence the nodal line of $\varphi$ divides a neighborhood of the origin into at least $6$ regions. Since $(0,0,-1)$ is also a critical point of $u$, the nodal line of $\varphi$ divides a neighborhood of infinity into at least $4$ regions, and therefore globally it divides $\mathbb{R}^2$ into at least 6 regions. Thus, an argument similar to that in the proof of Lemma \ref{threePointLemma} gives
\[
\frac{8\pi}{\alpha}\geq24\pi.
\]
This is impossible for $\alpha>\frac13$; when $\alpha=\frac13$, equality is excluded by the equality statement in the Sphere Covering Inequality because the right-hand side of \eqref{w} is strictly positive. Thus $w$ is symmetric with respect to the $y_1$-axis. By replacing $y_1$ axis with $y_2$ axis in the above argument, we conclude that $w$ is also symmetric with respect to the $y_2$-axis. The proof is now complete. \hfill $\Box$

\begin{proposition}
Let \( \frac{1}{3} \leq \alpha \leq 1 \), and let \( u \) be a solution of equation \eqref{mainPDE} satisfying $(\mathcal{H})$. 
If \( u \) is symmetric with respect to a plane passing through the origin, then it must be symmetric with respect to two orthogonal planes passing through the origin.
\end{proposition}

\noindent 
{\bf Proof.}
Let
\[
E:=\mathbb{S}^2\cap\{x_3=0\},
\]
and, without loss of generality, assume that \(u\) is symmetric with respect to the \(x_1x_2\)-plane. We prove that \(u\) is symmetric with respect to a plane orthogonal to it. Suppose, to the contrary, that no such symmetry plane exists.

Let \(x^*\) be a minimum point of \(u|_E\). By the assumed symmetry, the derivative of \(u\) normal to \(E\) vanishes at \(x^*\), while the tangential derivative vanishes because \(x^*\) is a minimum. Hence \(x^*\) is a critical point of \(u\).

We first claim that \(x^*\) is a nondegenerate minimum of \(u|_E\). Let \(y^*\) be a unit tangent vector to \(E\) at \(x^*\), let \(P_{y^*}\) be the plane through the origin with normal \(y^*\), and let \(R_{y^*}\) denote reflection with respect to \(P_{y^*}\). Consider
\[
\varphi(z):=u(z)-u(R_{y^*}z).
\]
If \(\varphi\equiv0\), then \(P_{y^*}\) is the desired second symmetry plane. Thus assume \(\varphi\not\equiv0\). Choose local coordinates \((x_1,x_2)\) centered at \(x^*\), with the \(x_1\)-axis tangent to \(E\) and the \(x_2\)-axis orthogonal to \(E\) in \(\mathbb{S}^2\). Then
\[
\varphi(-x_1,x_2)=-\varphi(x_1,x_2),
\qquad
\varphi(x_1,-x_2)=\varphi(x_1,x_2).
\]
Since \(x^*\) is a critical point, the vanishing order \(m\) of \(\varphi\) is at least \(3\), and its leading term is a nonzero harmonic homogeneous polynomial \(P_m\). The two parity relations imply that \(m\) is odd. If \(m=3\), then
\[
P_3(x_1,x_2)=c(x_1^3-3x_1x_2^2),\qquad c\neq0.
\]
However, if the minimum of \(u|_E\) were degenerate, its Taylor expansion at \(x^*\) would have neither a quadratic nor a cubic term; hence the restriction of \(\varphi\) to \(x_2=0\) could not have a nonzero cubic term. This excludes \(m=3\), so \(m\geq5\). Hence at least four nodal arcs enter each of the two hemispheres determined by \(P_{y^*}\), and the nodal graph yields at least three nodal domains in a hemisphere. The Sphere Covering Inequality argument used in Lemma~\ref{threePointLemma} then forces symmetry with respect to \(P_{y^*}\), a contradiction. Therefore \(x^*\) is nondegenerate.

Without loss of generality, we may assume that \( x^* = (1,0,0) \). Let \( x = (x_1,x_2,0) \) be close to \( x^* \) with \( x_2 < 0 \). Then
\[
\nabla u(x) \cdot y < 0,
\]
where \( y \) is the unit normal vector to the plane \(P_y\) passing through the origin and \(x\), orthogonal to the \(x_1x_2\)-plane, and oriented toward the half-space containing \(x^*\). Notice that \(-x^*\) cannot also be a critical point; otherwise Lemma~\ref{secondPlaneLemma2}, together with the assumed symmetry with respect to the \(x_1x_2\)-plane, would give the desired second symmetry plane. Hence, for \(x\) sufficiently close to \(x^*\), neither \(x\) nor \(-x\) is a zero of \(\nabla u(\,\cdot\,)\cdot y\) on \(C_y\). Therefore, by \((\mathcal{H})\) and the symmetry with respect to the \(x_1x_2\)-plane, there exist two distinct points \(p_x^1,p_x^2\in C_y\), reflected images of one another across the \(x_1x_2\)-plane, such that
\[
\nabla u(p_x^i)\cdot y=0,\qquad i=1,2.
\]

If there exists a third point \(p\in C_y\) such that
\[
\nabla u(p)\cdot y=0,
\]
then Lemma~\ref{threePointLemma} implies that \(u\) is symmetric with respect to \(P_y\), and the proof is complete.

We now consider the nodal set of
\[
\varphi_x(z):=u(z)-u(\widetilde z),
\]
where \(\widetilde z\) denotes the reflection of \(z\) with respect to the plane \(P_y\). There exists a domain \(\Omega_x\subset H_y\) such that \(\varphi_x<0\) in \(\Omega_x\), \(\varphi_x>0\) in \(H_y\setminus\Omega_x\), and \(\partial\Omega_x\) contains an arc with endpoints \(p_x^1\) and \(p_x^2\) passing through \(x\).

Since the nodal line of \(\varphi_{x^*}\) divides a neighborhood of \(x^*\) into at least six regions, it follows that
\[
p_x^i\to x^* \quad \text{and} \quad x\to x^*,\qquad i=1,2.
\]
If this nodal line touches \(C_y\) at any point other than \(x^*\), then it divides \(H_y\) into at least three regions, and Lemma~\ref{threePointLemma} again implies that \(u\) is symmetric with respect to \(P_y\).

On the other hand, for \(x=(x_1,x_2,0)\) with \(x_2>0\), we have
\[
\nabla u(x)\cdot y>0,
\]
where \(y\) is the unit normal vector to the plane \(P_y\) passing through the origin and \(x\), orthogonal to the \(x_1x_2\)-plane, and oriented toward the half-space not containing \(x^*\).  The preceding nodal-set argument shows that, after the two boundary zeros \(p_x^1,p_x^2\) coalesce at \(x^*\), no zero of \(\nabla u(\,\cdot\,)\cdot y\) can remain on \(C_y\) for \(x_2>0\) sufficiently small; otherwise, in view of hypothesis $(\mathcal{H})$,  the nodal set would acquire an additional intersection with \(C_y\), and Lemma~\ref{threePointLemma} would again imply symmetry with respect to \(P_y\). Since \(\nabla u(x)\cdot y>0\), it follows that
\[
\nabla u(z)\cdot y>0 \qquad \text{for all } z\in C_y.
\]
Hence \(\Omega_x\Subset H_y\), which contradicts assumption \((\mathcal{H})\). Thus \(u\) must be symmetric with respect to a plane orthogonal to the \(x_1x_2\)-plane passing through the origin.
\hfill $\Box$

\section{Axial symmetry} 
In this section we present the proof of Theorem \ref{mainTheorem} and show that for $\frac{1}{3}\leq \alpha \leq 1$, every solution of the mean field equation \eqref{mainPDE} satisfying $(\mathcal{H})$ must be axially symmetric. 

\begin{proposition}
Let \( \frac{1}{3} \leq \alpha \leq 1 \), and let \( u \) be a solution of Equation \eqref{mainPDE} satisfying \( (\mathcal{H}) \). If \( u \) is symmetric with respect to two orthogonal planes passing through the origin, then it must be axially symmetric. Moreover, if \( \alpha = \frac{1}{3} \), then \( u \equiv 0 \).
\end{proposition}
{\bf Proof.}  The case $\alpha=1$ is classical, so we may assume that $\alpha<1$.  Without loss of generality, we may assume that the two orthogonal symmetry planes are the $x_1x_3$- and $x_2x_3$-planes. Hence
\begin{align}\label{SymmetricTwoPlanes}
u(x_1,x_2,x_3)
=u(-x_1,x_2,x_3)
=u(x_1,-x_2,x_3),
\qquad (x_1,x_2,x_3)\in\mathbb{S}^2.
\end{align}
There exists \( p^* = (x_1^*, x_2^*, 0) \in \mathbb{S}^2 \) such that 
\[
\nabla u(p^*) \cdot k = 0,
\]
where \( k = (0, 0, 1) \).

If \( x_1^* = 0 \) or \( x_2^* = 0 \), then it follows from \eqref{SymmetricTwoPlanes} that \( \pm p^* \) are critical points of \( u \). By Lemma \ref{secondPlaneLemma2}, $u$ is symmetric with respect to two orthogonal planes containing the line through $\pm p^*$. The one among the two planes in \eqref{SymmetricTwoPlanes} that is orthogonal to $p^*$ is orthogonal to both of these planes. Hence $u$ is symmetric with respect to three mutually orthogonal planes through the origin.
Suppose \( x_1^* \neq 0 \) and \( x_2^* \neq 0 \). Then it follows from \eqref{SymmetricTwoPlanes} that 
\[
\nabla u(\pm x_1^*, \pm x_2^*, 0) \cdot k = 0.
\]
By Lemma \ref{threePointLemma}, $u$ is therefore symmetric with respect to the $x_1x_2$-plane, in addition to the two symmetries in \eqref{SymmetricTwoPlanes}. Thus, in either case $u$ is symmetric with respect to three mutually orthogonal planes. Since $\alpha<1$, Lemma 8 in \cite{Tian} applies and implies that \( u \) is axially symmetric. Finally, if \( \alpha = \frac{1}{3} \), it follows from Proposition 1 in \cite{Tian} that \( u \equiv 0 \).

\hfill \( \Box \)

\section*{Acknowledgments}

Changfeng Gui is supported by NSFC Key Program (Grant No. 12531010), University of Macau research grants CPG2024-00016-FST, CPG2025-00032-FST, SRG2023-00011-FST, MYRGGRG2023-00139-FST-UMDF,  and Macao SAR FDCT 0003/2023/RIA1 and Macao SAR FDCT 0024/2023/RIB1. Amir Moradifam is supported by NSF grant DMS-1953620.

\bibliographystyle{plain}
\bibliography{SphereCovering-2}

\end{document}